\documentclass[11pt]{article}
\usepackage[margin=1in]{geometry} 
\usepackage{amssymb,amsfonts,amsmath,bbm,mathrsfs,stmaryrd}
\usepackage{mathabx}
\usepackage{xcolor}
\usepackage{url}

\usepackage[colorlinks,
             linkcolor=black!75!red,
             citecolor=blue,
             pdftitle={fusion},
             pdfauthor={Alexandru Chirvasitu},
             pdfproducer={pdfLaTeX},
             pdfpagemode=None,
             bookmarksopen=true
             bookmarksnumbered=true]{hyperref}

\usepackage{tikz}
\usetikzlibrary{arrows,calc,decorations.pathreplacing,decorations.markings,intersections,shapes.geometric,through,fit,shapes.symbols,positioning,decorations.pathmorphing}

\usepackage[amsmath,thmmarks,hyperref]{ntheorem}
\usepackage{cleveref}

\crefname{section}{Section}{Sections}
\crefformat{section}{#2Section~#1#3} 
\Crefformat{section}{#2Section~#1#3} 

\crefname{subsection}{\S}{\S\S}
\crefformat{subsection}{#2\S#1#3} 
\Crefformat{subsection}{#2\S#1#3} 

\theoremstyle{plain}

\newtheorem{lemma}{Lemma}[section]
\newtheorem{proposition}[lemma]{Proposition}
\newtheorem{corollary}[lemma]{Corollary}
\newtheorem{theorem}[lemma]{Theorem}

\theoremstyle{nonumberplain}

\theoremstyle{plain}
\theorembodyfont{\upshape}
\theoremsymbol{\ensuremath{\blacklozenge}}

\newtheorem{definition}[lemma]{Definition}

\newtheorem{remark}[lemma]{Remark}

\crefname{definition}{definition}{definitions}
\crefformat{definition}{#2definition~#1#3} 
\Crefformat{definition}{#2Definition~#1#3} 

\crefname{ex}{example}{examples}
\crefformat{example}{#2example~#1#3} 
\Crefformat{example}{#2Example~#1#3} 

\crefname{remark}{remark}{remarks}
\crefformat{remark}{#2remark~#1#3} 
\Crefformat{remark}{#2Remark~#1#3} 

\crefname{convention}{convention}{conventions}
\crefformat{convention}{#2convention~#1#3} 
\Crefformat{convention}{#2Convention~#1#3}

\crefname{lemma}{lemma}{lemmas}
\crefformat{lemma}{#2lemma~#1#3} 
\Crefformat{lemma}{#2Lemma~#1#3} 

\crefname{proposition}{proposition}{propositions}
\crefformat{proposition}{#2proposition~#1#3} 
\Crefformat{proposition}{#2Proposition~#1#3} 

\crefname{corollary}{corollary}{corollaries}
\crefformat{corollary}{#2corollary~#1#3} 
\Crefformat{corollary}{#2Corollary~#1#3} 

\crefname{theorem}{theorem}{theorems}
\crefformat{theorem}{#2theorem~#1#3} 
\Crefformat{theorem}{#2Theorem~#1#3} 

\crefname{assumption}{assumption}{Assumptions}
\crefformat{assumption}{#2assumption~#1#3} 
\Crefformat{assumption}{#2Assumption~#1#3} 

\crefname{equation}{}{}
\crefformat{equation}{(#2#1#3)} 
\Crefformat{equation}{(#2#1#3)}

\theoremstyle{nonumberplain}
\theoremsymbol{\ensuremath{\blacksquare}}

\newtheorem{proof}{Proof}

\newcommand{\BH}{B(H)}
 \newcommand{\BK}{B(K)}

\newcommand\bC{\mathbb C}

\newcommand\bG{\mathbb G}

\newcommand\bZ{\mathbb Z}

\newcommand\cM{\mathcal M}

\newcommand\cU{\mathcal U}
\newcommand\cV{\mathcal V}

\newcommand\G{\Gamma}

\newcommand{\Aut}{\operatorname{Aut}}
\newcommand{\Inn}{\operatorname{Inn}}
\newcommand{\Out}{\operatorname{Out}}

\newcommand{\norm}[1]{\left\Vert#1\right\Vert}
\newcommand{\abs}[1]{\left\vert#1\right\vert}

\DeclareMathOperator{\out}{\mathrm{Out}}
\DeclareMathOperator{\aut}{\mathrm{Aut}}

\newcommand{\cat}[1]{\textsc{#1}}

\newcommand{\qedhere}{\mbox{}\hfill\ensuremath{\blacksquare}}

\numberwithin{equation}{section}

\title{Topological automorphism groups of compact quantum groups}
\author{Alexandru Chirvasitu\quad and\quad  Issan Patri}

\begin{document}

\date{}

\newcommand{\Addresses}{{
  \bigskip
  \footnotesize

  \textsc{Department of Mathematics, University
    of Washington, Seattle, WA 98195-4350, USA}\par\nopagebreak
  \textit{E-mail address}: \texttt{chirva@uw.edu}

  \medskip

  \textsc{Chennai Mathematical Institute, SIPCOT IT Park, Siruseri, Chennai, India}\par\nopagebreak
  \textit{E-mail address}: \texttt{ipatri@cmi.ac.in}
}}

\maketitle

\begin{abstract}
  We study the topological structure of the automorphism groups of
  compact quantum groups showing that, in parallel to a classical
  result due to Iwasawa, the connected component of identity of the automorphism
  group and of the ``inner'' automorphism group coincide.

  For compact matrix quantum groups, which can be thought of as quantum analogues of compact {\it Lie} groups, we prove that the
  inner automorphism group is a compact Lie group and the outer
  automorphism group is discrete. Applications of this to the study of
  group actions on compact quantum groups are highlighted.

  We end with the construction of a compact matrix quantum group whose
  fusion ring is not finitely generated, unlike the classical case. 
\end{abstract}


\vspace{.5cm}



\section*{Introduction}

Compact quantum groups \cite{Wor98} as introduced in \cite{Wor87} (at
the time under different terminology) are objects meant to capture the
concept of ``quantum symmetry'' for various geometric or algebraic
structures, and their construction in practice as quantum automorphism
groups of the latter reflects this perspective.

This procedure has proven very flexible and fruitful, and the
structures amenable to this type of treatment (and which hence give
rise to interesting examples of compact quantum groups) include finite
sets or possible non-commutative measure spaces \cite{Wan98}, finite
graphs \cite{Bic03}, metric spaces \cite{Ban05,Gos12} or Riemannian
manifolds \cite{BhoGos09,gj}, to name only a few examples.

Actions of quantum groups on operator algebras (often equipped with
various additional structures), as a generalization of actions of
groups on operator algebras, can be regarded as non-commutative
dynamical systems, with the dynamics in this case being encoded by the
quantum group action. Here, we propose to study compact quantum groups
themselves as dynamical objects, by investigating their own
(classical) automorphism groups.

This line of inquiry continues the work of \cite{iss,MP16} and in
pursuing it, we have several aims.

First, these investigations fit well with the point of view, prevalent
in the quantum group literature, that understanding the symmetries of
a mathematical object (in our case a compact quantum group) sheds
light on the underlying mathematical structure. For instance, the
orbit-boundedness-type result that constitutes \Cref{VInn} can be
regarded as witnessing a type of ``rigidity'' for sufficiently small
compact quantum groups (see also the discussion below, in the present
section). Besides, a good understanding of the classical automorphism group of a compact quantum group would give insights towards understanding ``quantum'' automorphism groups of compact quantum groups. Attempts have been made towards defining such a notion \cite{BSS15, KSW15, KSS16}.

Another source of motivation for us is the connection, in the
classical case, between the inner / outer distinction for
automorphisms of a compact group and the representation theory of the
latter; see e.g. \cite{Han93,ydv} and references in the latter. Since
compact quantum groups are essentially representation-theoretic
objects (being defined, in one incarnation, as Hopf algebras that
mimic algebras of matrix coefficients for finite-dimensional
representations), this connection seems ripe for investigation in the
quantum setting.

Finally, dynamics on compact groups is a rich topic in ergodic
theory. Though we cannot possibly do it justice in this short
introduction (see e.g. \cite{hlm,ktz,ktch,ktch-sch,jw95,jw12,htz} and
references cited there), we hope that the present work draws some
analogies between that field and an incipient quantum analogue thereof
by regarding compact quantum groups as non-commutative spaces being
acted upon by structure-preserving transformations.

The paper is organized as follows.

In \Cref{se.prel} we recall some of the background needed later, such
as definitions, conventions, and so on.

\Cref{se.auto} is devoted to the main results of the paper. We
introduce for each compact quantum group $G$ its group of
automorphisms $\Aut(G)$, and study the topology and structure of the
latter. In first instance, we will see
(\Cref{pr.same-top,re.tops,cor.same-top}) that there is essentially
one reasonable way of topologizing $\Aut(G)$.

As for structure theory, we express $\Aut(G)$ as an extension of the
{\it outer} automorphism group $\out(G)$ by the group $\Inn(G)$ of
{\it inner} automorphisms and prove in \Cref{th.iwa} that the full
automorphism group and the inner one share the same connected
component of the identity. Moreover, \Cref{th.out} shows that for
compact matrix quantum groups (typically regarded as non-commutative
analogues of compact Lie groups) the outer automorphism group is
discrete; this mimics the case of classical compact Lie groups.

In \Cref{se.appl} we apply some of the contents of the preceding
section to the study of dynamics on compact quantum groups, the main
result (\Cref{VInn}, paraphrased here) being that if a discrete group acts
{\it compactly} on a compact matrix quantum group $G$, the resulting
action on the dual object of irreducible representations of $G$ has
uniformly bounded orbits. Next, in \Cref{Outfin} using dynamical methods, we show that for a large family of compact matrix quantum groups, the outer automorphism group is finite and the automorphism group is a compact Lie group.

 Finally, in \Cref{se.counter} we present some examples of compact
 matrix quantum groups whose tensor categories of representations have
 infinitely-generated representation rings. This is in contrast to the classical case of compact Lie groups, for which it is well know that the representation rings are finitely generated \cite[Corollary~3.3]{Se68}.

 The relevance of this latter topic to the present work is that the
 question of whether or not such examples exist is motivated by the
 structure theory of the automorphism groups of compact matrix quantum
 groups (see \Cref{fin}).

\subsection*{Acknowledgements}
This work was initiated at the 7th ECM satellite conference ``Compact
Quantum Groups'' at Greifswald, Germany and the authors are grateful
to the organizers, Uwe Franz, Malte Gerhold, Adam Skalski and Moritz
Weber, for the invitation. The authors would also like to thank Yuki
Arano and Makoto Yamashita for interesting discussions.

The first author was partially supported by NSF grant DMS - 1565226.

\section{Preliminaries}\label{se.prel}

In this section, we give a quick introduction to compact quantum groups and CQG algebras. We refer to \cite{Wor87}, \cite{Wor98} and \cite{DK} for more details.

A compact quantum group $G = (A_{t}, \Phi)$ is a unital $C^{\ast}$-algebra  $A_t$ together with a comultiplication 
 $\Delta: A_t \to A_t \otimes A_t$ which is a coassociative $*$-homomorphism,  
such that the closure of the linear span of the sets $(A_t \otimes 1)\Delta(A_t)$ and $(1 \otimes A_t)\Delta(A_t)$ 
are dense in $A_t \otimes A_t$ in norm.  Let $H$ be a finite dimensional Hilbert space with an orthonormal basis given by 
$\{e_1, e_2, \cdots, e_n\}$ and with $e_{ij}$, $1\leq i,j\leq n$, the corresponding system of matrix units in 
$B(H)$. A unitary $u=\sum_{i,j}e_{ij}\otimes u_{ij}\in B(H)\otimes A_t$ 
is said to be a finite dimensional representation of $G$, if 
$\Delta(u_{ij})=\sum_k u_{ik}\otimes u_{kj}$ for all $1\leq i,j\leq n$. Given two finite-dimensional representations $u\in \BH\otimes A_t$ and $v\in \BK\otimes A_t$, an operator $T\in B(H,K)$ is said to be an intertwiner of $u$ and $v$ if $(T\otimes id)u=v(T\otimes id)$. Two 
representations $u\in \BH\otimes A_t$ and $v\in \BK\otimes A_t$ are said to be 
equivalent if there exists an unitary intertwiner between them. A representation is said to be irreducible if the only self-intertwiners are the trivial ones. We denote the set of equivalence classes of irreducible representations of $G$ by Irr$(G)$. 

A compact matrix quantum group is a compact quantum group $G$ for which there exists a finite dimensional representation $u\in B(H)\otimes A_t$, with $u=\sum_{i,j}e_{ij}\otimes u_{ij}$, such that the elements $u_{ij}$ generate $A_t$ as a $C^\ast$-algebra. This is a generalization of the notion of compact lie groups (since for a compact lie group $G$, there always exists a faithful finite dimensional representation and thus, its coefficient functions generate $C(G)$ as a $C^\ast$-algebra).

Just like in the classical case of compact groups, for the case of a compact quantum group $G$ as well, there exists a canonical Hopf $\ast$-algebra $A$ which is a dense subalgebra of $A_t$ and is spanned by the matrix coefficients of the irreducible representations of $G$. This Hopf $\ast$-algebra is referred to as the CQG algebra associated to $G$. If $G$ is a compact matrix quantum group, we will refer to $A$ as a CMQG algebra. It is also possible to traverse the converse route, whence starting with a CQG algebra $A$, which is defined as a Hopf algebra with some additional properties, one can then go on to associate a compact quantum group $G$ to this algebra \cite{DK}. 

We now consider an example, which will be important for us in the sequel. Let $n\ge 2$ be a positive integer, and recall the Hopf $*$-algebra
$A_u(n)$ that is the function algebra of the so-called free unitary
group $U_n^+$, as introduced in \cite{Wan95_1} (and in more general
form in \cite{DaeWan96}). As a $*$-algebra, it is generated freely by
$u_{ij}$ for $1\le i,j\le n$ subject to the relations making
\begin{equation*}
  (u_{ij})_{i,j},\ (u_{ij}^*)_{i,j}\ \in\ M_n(A_u(n))
\end{equation*}
unitary.

The $*$-structures will not play a major role here, owing to the fact
that Hopf subalgebras and quotient Hopf algebras of CQG algebras
(\cite{DK}) are automatically CQG subalgebras and quotients
respectively. In other words, so long as morphisms preserve the Hopf
structures, the $*$ structures come along for the ride.

\section{Automorphism groups of compact quantum groups}\label{se.auto}

In this section we study the topological structure of the automorphism
group of compact quantum groups in general and of compact matrix
quantum groups in particular. Before we move on to the main results of
the section, let us recall some background on the topology of the
automorphism group of a compact quantum group.

Let $A$ be a CQG algebra underlying a compact quantum group
$G$. Recall that $\Aut(G)$ (or $\Aut(A)$) is defined in \cite{iss}
as the group of coproduct-preserving automorphisms of the full
enveloping $C^*$-algebra $\overline{A}$. Such automorphisms of compact quantum groups were first studied by Wang \cite{Wan95_2}. As noted there, restricting
automorphisms to the dense Hopf $*$-subalgebra $A$ of $\overline{A}$
induces a bijection onto the group of Hopf $*$-algebra automorphisms
of $A$. Moreover, this also bijects onto the group of
coproduct-preserving automorphisms of the {\it reduced} $C^*$ envelope
$\overline{A}_r$ of $A$ (see discussion preceding \cite[Lemma
3.3]{iss}). We will henceforth freely identify these groups.

Recall furthermore that $\Aut(A)$ is topologized in \cite{iss}
(discussion preceding Theorem 3.4) via the identity neighborhoods
\begin{equation}\label{eq:def-top}
  U(a_1,\cdots,a_n;\varepsilon) = \{\alpha\in Aut(A)\ |\ \|\alpha(a_i)-a_i\|<\varepsilon,\ \forall i\},
\end{equation}
where the norm used is that on $\overline{A}$.

\begin{remark}
  This is initially defined for tuples $\{a_i\}$ in $\overline{A}$,
  but \cite[Lemma 3.5]{iss} ensures it is enough to consider elements
  of the dense $*$-subalgebra $A$ of $\overline{A}$.
\end{remark}

We will take here a slightly different approach to topologizing
$\Aut(A)$, which will make it clear that in \Cref{eq:def-top} we may
as well have used the reduced $C^*$-norm of $\overline{A}_r$, or the
$L^2$ norm associated to the Haar state, or the $L^p$ norm for some
$p\ge 1$, etc. In other words, $\Aut(G)$ is, not just as an abstract group, but even as a topological group, is determined by the quantum group structure of $G$, independent of the norm completion.

Consider an automorphism $\alpha$ of the Hopf $*$-algebra $A$. It will
permute the simple subcoalgebras of $A$ and hence induce a permutation
of the set Irr$(G)$ of irreducible $A$-comodules. This provides a morphism
\begin{equation*}
  \Aut(A) \to S_{\textrm{Irr}(G)},
\end{equation*}
where the latter symbol denotes the symmetric group on the set 
Irr$(G)$. The kernel of this homomorphism consists of precisely
those automorphisms of $A$ that preserve the class of each irreducible
$\bG$-representation, i.e. the normal subgroup
\begin{equation*}
  \Aut_\chi(A)\trianglelefteq \Aut(A)
\end{equation*}
of {\it inner} automorphisms of \cite{iss} (definition preceding
Theorem 3.4). Note that every element of $\Aut_\chi(A)$ is uniquely
determined by its effect on the simple subcoalgebras of $A$. Since the
automorphism group of each such $n\times n$ matrix subcoalgebra is the
projective $n\times n$ linear group, we have an embedding
\begin{equation*}
  \Aut_\chi(A)\subset \prod_{V\in \textrm{Irr}(G)}PGL(V). 
\end{equation*}
All in all, we obtain an embedding
\begin{equation}\label{eq:wr}
  \Aut(A)\subseteq \left(\prod_{V\in\textrm{Irr}(G)}PGL(V)\right)\rtimes S_{\textrm{Irr(G)}}. 
\end{equation}
The right hand side of \Cref{eq:wr} has a natural topological group
structure defined as the product topology on the set
\begin{equation*}
  \left(\prod PGL(V)\right)\times S_{\textrm{Irr}(G)},
\end{equation*}
where the projective general linear groups have their standard Lie
group topologies and the symmetric group $S_{\textrm{Irr}(G)}$ is
topologized via the compact-open topology resulting from the discrete
topological space structure of $\textrm{Irr}(G)$ (this, for instance, is
the topology on a possibly infinite symmetric group used in
\cite[$\S$4]{bs}).

The following result ensures that the topology on $\Aut(A)$ defined
via the neighborhoods \Cref{eq:def-top} is unambiguous with respect to
the norm used therein by identifying it with the topology inherited
from the embedding \Cref{eq:wr}. 

\begin{proposition}\label{pr.same-top}
  Let $A$ be a CQG algebra, and $\|\cdot\|$ a norm on $A$ making the
  latter into a normed complex vector space such that the Haar state
  $h:A\to \bC$ is continuous. Then, the topology on $\Aut(A)$ defined
  via the identity neighborhoods
  \begin{equation}
    \label{eq:def-top-bis}
    U(a_1,\cdots,a_n;\varepsilon)  = \{\alpha\in \Aut(A)\ |\ \|\alpha(a_i)-a_i\|<\varepsilon,\ \|\alpha^{-1}(a_i)-a_i\|<\varepsilon, \forall a_i\}
  \end{equation}
  for various tuples $\{a_i\}\subset A$ and $\varepsilon>0$ coincides
  with the subspace topology inherited from \Cref{eq:wr}. 
\end{proposition}
\begin{proof}
  Let us denote the topology coming from \Cref{eq:wr} by
  $\cat{top}_{emb}$ (for {\it embedding}) and the topology defined via
  \Cref{eq:def-top-bis} by $\cat{top}_{\|\cdot\|}$. We want to argue
  that each contains the other. Since both make $\Aut(A)$ into a
  topological group, this is equivalent to proving that open identity
  neighborhoods with respect to each topology contain open identity
  neighborhoods with respect to the other; we will do this separately
  for the two inclusions.

  \vspace{.5cm}

  {\bf (1):} $\cat{top}_{emb}\subseteq \cat{top}_{\|\cdot\|}$. Let
  $U\subseteq \Aut(A)$ be an identity neighborhood in
  $\cat{top}_{emb}$. Fix a simple $A$-comodule $V$, and denote by
  $C\subset A$ its coefficient coalgebra. Then we may as well assume
  $U$ is of the form
  \begin{equation*}
    U=\{\alpha\in \Aut(A)\ |\ \alpha(C)=C,\ \alpha|_C\in \cU\subset PGL(V)\cong \Aut(C)\},
  \end{equation*}
  where $\cU$ is a neighborhood of the identity in the Lie group
  $PGL(V)$. Indeed, intersections of such open identity neighborhoods
  in $\cat{top}_{emb}$ form a basis (around $1\in \Aut(A)$) for the latter
  topology.

  Let $\chi_V\in C$ be the character of $V$. Then we claim that if
  $\varepsilon>0$ is small enough, then $U(\chi_V;\varepsilon)$
  (defined as in \Cref{eq:def-top-bis}) consists of elements
  $\alpha\in\Aut(A)$ with $\alpha(C)=C$ (or equivalently,
  $\alpha(\chi_V)=\chi_V$). Indeed, $\alpha(\chi_V)$ is a character
  $\chi_W$ for some irreducible $A$-comodule $W$, and the conclusion that $W\cong V$ for sufficiently small $\varepsilon$ follows from the continuity of $h$ and the fact that
  \begin{equation*}
    h(\chi_V\chi_W^*)\in \{0,1\},
  \end{equation*}
  with $1$ if $W\cong V$ and $0$ otherwise.

  Finally, having chosen $\varepsilon$ small enough to ensure that for any $\alpha\in U(\chi_V;\varepsilon), \alpha(C)=C$, we can shrink it further to make sure that for a basis $u_{ij}$, $1\le i,j\le n$ of $C$ the restriction of every
  \begin{equation*}
    \alpha\in U(u_{ij},\ 1\le i,j\le n; \varepsilon)
  \end{equation*}
  to $C$ is contained in the neighborhood $\cU\subset PGL(V)$. Indeed,
  this follows from the fact that since $C$ is finite-dimensional, all
  topologies making it into a complex topological vector space
  coincide.

  \vspace{.5cm}

  {\bf (2):} $\cat{top}_{\|\cdot\|}\subseteq \cat{top}_{emb}$. It is
  sufficient to select an open set of the form \Cref{eq:def-top-bis}
  with $n=1$, and $a=a_1\in A$ belonging to some simple subcoalgebra
  $C\subset A$. Clearly then, for a small enough neighborhood $\cU$ of
  the identity in $PGL(C)$, the set
  \begin{equation*}
    \{\alpha\in \Aut(A)\ |\ \alpha(C)=C,\ \alpha|_C\in \cU\subset PGL(V)\cong \Aut(C)\}
  \end{equation*}
  used in the proof of part (1) will be contained in
  $U(a;\varepsilon)$.

  This finishes the proof of part (2) and of the proposition.
\end{proof}

\begin{remark}\label{re.tops}
  The proof of \Cref{pr.same-top} given above extends to prove a
  stronger statement applying to topologies on $\Aut(A)$ defined analogously to \Cref{eq:def-top-bis}, via neighborhoods of the form
  \begin{equation*}
    U(a_i;\cU) = \{\alpha\in \Aut(A)\ |\ \alpha(a_i)-a_i\in \cU \ni \alpha^{-1}(a_i)-a_i,\ \forall i\}
  \end{equation*}
  where $\cU$ is a neighborhood of $0\in A$ with respect to some
  topological vector space structure on $A$ making $h:A\to \bC$
  continuous.
\end{remark}

We now record a consequence of \Cref{pr.same-top}. As noted above,
given a compact quantum group we can also define the group $\Aut_r(A)$
of coproduct-preserving automorphisms of the reduced $C^*$-algebra
$\overline{A}_r$. We understand $\Aut_r(A)$ to be topologized by the
identity neighborhoods \Cref{eq:def-top-bis} for the reduced
$C^*$-norm $\|\cdot\|$.

Since any $\alpha\in\Aut(A)$ preserves the Haar state, it defines a
unique element in $\Aut_r(A)$. Hence, we have a canonical map
$\Aut(A)\rightarrow \Aut_r(A)$ which is clearly a group
homomorphism. It also easy to show that this map is bijective. The
inverse bijection is constructed in the following way. Since any
$\alpha\in\Aut_r(A)$ restrict to an automorphism of the dense Hopf
$*$-algebra $A$, it extends uniquely by the universal property to an
automorphism in $\Aut(A)$.

Similarly, since any automorphism of $\overline{A}_r$ intertwining
$\Delta$ has a unique normal extension to ${\rm L}^\infty(G)$, it
induces a bijective map $\Aut_r(A)\rightarrow \Aut_\infty(A)$, where
the latter is the set of comultiplication-preserving automorphisms of
the reduced von Neumann algebra $L^\infty(G)$. We topologize
$\Aut_\infty(A)$ by neighborhoods analogous to \Cref{eq:def-top-bis}.

Now, \Cref{pr.same-top} proves

\begin{corollary}\label{cor.same-top}
  For any compact quantum group $G$ with underlying CQG algebra $A$
  the groups $\Aut(A)$, $\Aut_r(A)$ and $\Aut_\infty(A)$ are
  isomorphic as topological groups. \qedhere
\end{corollary}

As mentioned above, the problem of seeking CMQGs whose fusion algebra is not finitely generated was prompted by the question of whether or not the outer automorphism group $\out(A)$ of a CMQG algebra $A$ is discrete. In fact, for a CMQG $G$, if the fusion ring $Z(G)$ is finitely generated, then the group $\out(G)$ is discrete. This is not hard to see and was exhibited in \cite{MP16}. We recap the proof here for the convenience of the reader.

\begin{lemma}\label{Open}
Let $[u]\in \text{\emph{Irr}}(G)$ and let $\chi(u)$ denote the character of $u$. The stabilizer subgroup of $u$ defined by 
$$K_{u}:= \{\alpha\in {\Aut}(\G) : \alpha(\chi(u))=\chi(u)\}$$
is an open subgroup of $\Aut(G)$.
\end{lemma}

\begin{proof}
Consider the open set defined by  $V(\chi(u), 1)= \{\alpha\in\Aut(G): \norm{\chi(u) - \alpha(\chi(u))}<1\}$. Note that $h(\chi(u)^\ast \chi(u)) = 1$, and for $\alpha\in Aut(G)$ one has $h(\alpha(\chi(u))^\ast\chi(u))=1$ if $\alpha\in K_u$ and zero otherwise. So, for $\alpha \in V(\chi(u), 1)$, we have
$$\abs{h(\chi(u)^\ast \chi(u)) - h(\alpha(\chi(u))^\ast \chi(u))} \leqslant \norm{\chi(u) - \alpha(\chi(u))}<1$$
and hence $\abs{1 - h(\alpha(\chi(u))^\ast \chi(u))} < 1$. It follows that $\alpha \in K_u$.
Consequently, $K_u$ is indeed an open subgroup of $\Aut(G)$.
\end{proof}

Using the above lemma, it was shown in \cite{iss} that the group
$\Out(G)$ for any compact quantum group $G$ is totally
disconnected.

\begin{proposition}\label{fin}
  Suppose the fusion ring $Z(G)$ of $G$ is finitely generated as a
  ring. Then $\Out(G)$ is discrete, equivalently $\Aut_\chi(G)$
  is an open subgroup of $\Aut(G)$.
\end{proposition}

\begin{proof}
Denote the generators of $Z(G)$ by $\lambda_1,\lambda_2,\cdots,\lambda_n$. The characters of irreducible representations of $G$ form a basis of $Z(G)$. Define the set 
\begin{align*}
N:=\{\chi(u) : \chi(u) \text{ appears in the linear decomposition of some }\lambda_i, 1\leq i\leq n\},
\end{align*}
where $[u]\in \text{Irr}(G)$. Clearly, $N\subseteq G_{char}$ is finite, where $G_{char}$ deotes the set of characters of $G$. Note that $\alpha\in \Aut(G)$ is in $\Aut_\chi(G)$ if and only if $\alpha(\chi(u))=\chi(u)$ for all $\chi(u)\in N$. 
Indeed, observe that $\alpha(\lambda_i)=\lambda_i$ for all $1\leq i\leq n$ if $\alpha(\chi(u))=\chi(u)$ for all $\chi(u)\in N$. But since $\lambda_i$ generate $Z(G)$ $($as a ring$)$, we have $\alpha(a)=a$ for all $a\in Z(G)$ if $\alpha(\chi(u))=\chi(u)$ for all $\chi(u)\in N$. Since $G_{char}\subseteq Z(G)$ it follows that $\alpha\in Aut_\chi(G)$. The other implication is obvious.

To complete the proof, fix $[u]\in \text{Irr}(G)$ and consider the
stabilizer group $K_u$ defined in \Cref{Open}. Write
$\mathcal{O}=\{[u]\in \text{Irr}(G): \chi(u)\in N\}$. Observe that
$\Aut_\chi(G)=\cap_{[u]\in \mathcal{O}}K_{u}$. Since $N$ is finite use
\Cref{Open} to finish the proof.
\end{proof}

Thus, finite generation of the fusion ring for a compact matrix quantum group would have provided an affirmative answer to the question whether its outer automorphism group is discrete, which is true in the classical case of compact Lie groups, hence the relevance of that problem. However, we show that there exist examples of compact matrix quantum for which fusion ring is not finitely generated as a ring. However, inspite of such examples (see \Cref{se.counter} below), we can show that the outer automorphism group of a $CMQG$ is indeed discrete. This illustrates an interesting phenomena in the theory of quantum groups, where the examples are diverse enough for such examples to exist, but at the same time, the theory is nice enough to allow us to prove unified results for all quantum groups.

\begin{theorem}\label{th.out}
  For any CMQG $A$ the group $\out(A)$ is discrete. 
\end{theorem}
\begin{proof}
  Let 
  \begin{equation*}
    C=\mathrm{span}\{u_{ij},\ 1\le i,j\le n\}
  \end{equation*}
be a subcoalgebra generating $A$ as an algebra, and corresponding to some $A$-comodule $U$ with basis $e_i$, with coaction
\begin{equation*}
  e_j\mapsto \sum_i e_i\otimes u_{ij}. 
\end{equation*}
We will show that the group $\aut_\chi(A)$ of {\it inner} automorphisms of $A$ is open. Equivalently, we have to show that if an automorphism $\alpha$ of $A$ is sufficiently close to the trivial automorphism $1\in \aut=\aut(A)$ in the standard topology on $\aut$ then it is inner. 

Sufficiently close to $1$ certainly implies that $\alpha$ preserves the class of $U$, and hence acts on the coalgebra $C$. This means in particular that some open neighborhood of $1\in\aut$ is contained in the isotropy group $\aut_U$ of (the class of) $U$. 

Every element in $\aut_U$ acts on the coalgebra $C$, and moreover,
since the latter generates $A$ as an algebra, we have an {\it
  embedding} $\aut_U\to \aut(C)$. The latter is a closed Lie subgroup
of $PGL_n(\bC)$ (the latter being the automorphism group of an
$n\times n$ matrix coalgebra), and moreover the inclusion
\begin{equation*}
  \aut_U\subseteq \aut(C)\subseteq PGL_n(\bC)
\end{equation*}
is easily seen to be closed (see \Cref{le.out-aux} below). In
conclusion $\aut_U$ is a Lie group. But this then implies that some
neighborhood $\cV$ of the identity in $\aut_U$ (and hence also in
$\aut$) is a (path-)connected Lie group.

Now consider the action of elements $\beta$ in the neighborhood $\cV$
on some fixed simple subcoalgebra $D\subset A$. The image $\beta(D)$
is a simple subcoalgebra of $A$ that either coincides with $D$, or is
orthogonal to $D$ with respect to the Haar state. The fact that
$1\in\aut$ fixes $D$ and can be connected to $\beta$ by a path rules
out the latter possibility, and hence $\beta(D)=D$.

It follows from the previous paragraph that $\beta$ fixes the class of
the irreducible comodule corresponding to $D$; since this comodule was
chosen arbitrarily from among the irreducible representations of the
quantum group, we obtain the desired conclusion that $\beta$ acts
trivially on the representation ring of $A$.
\end{proof}

\begin{lemma}\label{le.out-aux}
  Let $A$ be a normed complex algebra and $C\le A$ a finite-dimensional
  linear subspace. Then, the embedding
  \begin{equation*}
    \aut(A,C)\subseteq PGL(C)
  \end{equation*}
is closed. 
\end{lemma}
\begin{proof}
  The elements of $\aut(C)$ that arise by restriction from $\aut(A,C)$
  are precisely those that preserve all relations between elements of
  $C$. But the fact that $A$ is a topological algebra implies that the
  preservation of each relation is a closed condition, and the
  conclusion follows.
\end{proof}

\begin{corollary}\label{cor.lie}
  For any compact matrix quantum group $G$, the automorphism group
  $\Aut_\chi(G)$ is a compact Lie group.
\end{corollary}
\begin{proof}
  First let us note that $\Aut_\chi(G)$ is compact (see Theorem 3.4 of
  \cite{iss}). Now, as shown in the proof of \Cref{th.out}, the group
  $\Aut_U$ is a Lie group. It is easy to see that $\Aut_\chi(G)$ is a
  closed subgroup of $\Aut_U$ and hence is itself a Lie group.
\end{proof}

\begin{corollary}
  For any compact matrix quantum group $G$, the automorphism group
  $\Aut(G)$ is a Lie group.
\end{corollary}

\begin{proof}
  For a compact matrix quantum group $G$, the normal subgroup
  $\Aut_\chi(G)\trianglelefteq \Aut(G)$ (which is compact and Lie
  according to \Cref{cor.lie}) is open by \Cref{th.out}. In conclusion
  $\Aut(G)$ is an extension of a discrete group by a Lie group, and is
  hence Lie.
\end{proof}

We now proceed to prove an Iwasawa-type theorem for quantum groups. One
of the main results of the seminal paper \cite{Iwa} states that for a
compact group $G$, the connected component of the identity (henceforth
connected component) of the automorphism group of $G$ equals the
connected component of the inner automorphism group of $G$.

Let now $G$ be a compact quantum group. Let us denote the connected
component of $\Aut(G)$ by $\Aut^0(G)$ and the connected component of
the group $\Aut_\chi(G)$ by $\Aut_\chi^0(G)$.

\begin{theorem}\label{th.iwa}
 For any compact quantum group $G$, we have $\Aut^0(G)\cong \Aut^0_\chi(G)$.
\end{theorem}

\begin{proof}
  It is obvious that $\Aut^0_\chi(G)\subseteq \Aut^0(G)$. So, the
  theorem will follow if we can show that if $\alpha\in \Aut^0(G)$,
  then $\alpha\in \Aut^0_\chi(G)$. To this end, we fix an irreducible
  representation $u^a$ of $G$, with character denoted by $\chi_a$ and
  define the following continuous function
 $$f_a:\Aut(G)\rightarrow \{0,1\}$$
 $$\alpha\mapsto h_G(\chi_a^\ast\alpha(\chi_a))$$
 It is easy to see that $f_a$ is a continuous function. Now, when
 $\alpha$ is the trivial automorphism, then we have that
 $f_a(\alpha)=1$, hence, we have, by connectedness of $\Aut^0(G)$,
 that for any $\alpha\in \Aut^0(G)$, $f_a(\alpha)=1$, which implies
 that $\alpha(\chi_a)=\chi_a$. Since, the choice of the irreducible
 representation was arbitrary, we have $\alpha\in \Aut_\chi(G)$, and
 hence, the result follows.
\end{proof}

This then allows us to prove a straightforward corollary.

\begin{corollary}
Let $G$ be a compact quantum group. Then the connected component of the group $\Aut(G)$ is a compact group. If $G$ is a compact matrix quantum group, then the connected component of $\Aut(G)$ is a compact Lie group.
\end{corollary}

\begin{proof}
  This follows easily from the previous theorem and from the
  compactness of the group $\Aut_\chi(G)$ (see Theorem 3.4 of
  \cite{iss}). In the compact matrix quantum group case, this follows
  from the previous theorem and \Cref{cor.lie}.
\end{proof}

In fact, one can prove stronger structural results about the group $\Aut(G)$ if a little more is known about the compact quantum group $G$.  

\begin{proposition}
Let $G$ be a compact matrix quantum group such that $\out(G)$ is finite. Then $\Aut(G)$ is compact Lie group.
\end{proposition}

\begin{proof}
Since $\out(G)=\Aut(G)/\Aut_\chi(G)$, and since $\Aut_\chi(G)$ is a compact Lie group in case that $G$ is a compact matrix quantum group, hence, when $\out(G)$ is finite, it follows easily that $\Aut(G)$ is a compact Lie group. 
\end{proof}

\begin{corollary}
Let $G$ be a compact quantum group with fusion rules identical to that of simple compact Lie group. Then $\Aut(G)$ is a compact Lie group.
\end{corollary}

\begin{proof}
This follows from the previous proposition and from the fact that for a compact quantum group $G$ with fusion rules isomorphic to that of a simple compact Lie group, the group $\out(G)$ can only have order $1,2,3$ or $6$ (see Proposition 3.8 of \cite{iss}). 
\end{proof}

In order to state the following auxiliary result that has come up in
the above proof, we place ourselves in the following context: $A$ is a
topological complex algebra, $C\le A$ is a finite-dimensional
linear subspace that generates $A$, and $\aut(A,C)$ is the group of
automorphisms of $A$ that leave $C$ invariant.

Under the above setup, restricting elements of $\aut(A,C)$ to $C$
gives rise to an inclusion
\begin{equation*}
  \aut(A,C) \subseteq \aut(C)\cong PGL(C). 
\end{equation*}
As before, we topologize the group $\aut=\aut(A)$ pointwisely using
the topology of $A$: the neighborhoods of $1\in\aut$ are given by 
\begin{equation}\label{eq:top}
  \{g\in\aut\ |\ ga_i-a_i,\ g^{-1}a_i-a_i\ \in\ U,\ \forall i\}
\end{equation}
for various choices of neighborhoods $U\subseteq A$ of $0$ and finite
sets $\{a_i\}\subset A$.

\begin{remark}
  We have included both $g$ and $g^{-1}$ in \Cref{eq:top} in order to
  ensure that $\aut$ is a topological {\it group} and hence the
  inverse is continuous. 

  The omission of $g^{-1}$ in the analogous definition from \cite{iss}
  (see the discussion preceding Theorem 3.4 therein) makes no
  difference, since in that case the action is by isometries with
  respect to a $C^*$ norm on $A$. 
\end{remark}

\section{Some applications to CQG dynamical systems}\label{se.appl}

In this section we use the results of last section, to draw some
quick conclusions about spectral properties of CQG dynamical systems
in certain cases.

In the paper \cite{MP16}, the notion of CQG dynamical systems was introduced and studied, which is a tuple $(G,\Gamma,\alpha)$, where $G$ is a compact quantum group and $\Gamma$ acts on $C(G)$ by quantum group automorphisms, with $\alpha$ denoting the action. Thus, we have a group homomorphism $\alpha:\Gamma\rightarrow \Aut(G)$. 

\begin{definition}
Let $(G,\Gamma,\alpha)$ be a CQG dynamical system. Let $\norm{\cdot}$ denote the $C^\ast$-norm on $A$.
\begin{enumerate}
\item  \cite{Av} We say that the action is almost periodic if given any $a\in A$, the set $\{\alpha_{\gamma}(a): \gamma\in \Gamma\}$ is relatively compact in $A$ with respect to $\norm{\cdot}$.
\item We say that the action is compact if given any $a\in A$, the set $\{\alpha_{\gamma}(a)\Omega_{h}: \gamma\in \Gamma\}$ is relatively compact in $L^{2}(A)$ with respect to the $\norm{\cdot}_{2,h}$.
\item The extended action of $\Gamma$ on $L^{\infty}(G)$ is compact if given any $a\in L^{\infty}(G)$, the set $\{\alpha_{\gamma}(a)\Omega_{h}: \gamma\in \Gamma\}$ is relatively compact in $L^{2}(A)$ with respect to the $\norm{\cdot}_{2,h}$.
\end{enumerate}
\end{definition}

The following theorem was proved in \cite{MP16} (see Theorem 3.9)
\begin{theorem}\label{Cpt}
Let $(G,\Gamma,\alpha)$ be a CQG dynamical system. TFAE:\\
$(i)$ The action is almost periodic$;$\\
$(ii)$ the action is compact$;$ \\
$(iii)$ the orbit of any irreducible representation in \emph{Irr}$(G)$ is finite$;$\\
$(iv)$ the extended action of $\Gamma$ on $L^{\infty}(G)$ is compact.
\end{theorem}

It then follows immediately that if the action is ``inner'' in the sense that $\alpha(\Gamma)\subset \Aut_\chi(G)$, then the corresponding CQG dynamical system $(G,\Gamma,\alpha)$ is compact. This is because it follows from the definition of $\Aut_\chi(G)$ that the induced action of $\Gamma$ on Irr$(G)$ is trivial. In this context, a pertinent question to ask is whether one can obtain a result converse to this. In other words, suppose we have a compact CQG dynamical system, to what extent is that action ``inner''. It was observed in \cite{MP16}, that for a compact quantum group $G$, with $\out(G)$ discrete, any compact CQG dynamical system is virtually inner. Thus, in case $G$ is a compact matrix quantum group, we have the following

\begin{theorem}\label{VInn}
Let $G$ be a compact matrix quantum group. Let $(G,\Gamma,\alpha)$ be a compact CQG dynamical system. Then the subgroup  
$$\Gamma_{\chi}:= \{\gamma\in \Gamma : \alpha_\gamma\in \Aut_{\chi}(G)\}$$
of $\Gamma$ is of finite index.
\end{theorem}

\begin{proof}
  Recall that any $C^\ast$-dynamical system $(A,\Gamma,\alpha)$ is
  almost periodic if and only if the closure of the image of $\Gamma$
  under $\alpha$ in $\Aut(A)$ is compact in the pointwise norm
  topology $($see Theorem 2.2 and Corollary 2.7 of \cite{Gr}$)$. Thus,
  since $(G,\Gamma,\alpha)$ is compact, it follows that
  $H:=\overline{\Gamma}\subseteq \Aut(G)\subseteq \Aut(A)$ is
  compact. But we have, by \Cref{th.out} $\Out_\chi(G)$ is discrete,
  equivalently $\Aut_\chi(G)$ is an open subgroup of $\Aut(G)$. So,
  $H_\chi=\Aut_\chi(G)\cap H$ is open in $H$. Since $H$ is compact, so
  $H_\chi$ is of finite index in $H$. Consequently, $\Gamma_\chi$ is
  finite index subgroup of $\Gamma$.
\end{proof}

The next result uses the ideas of CQG dynamical systems to give a structural result on automorphism groups.

\begin{theorem}\label{Outfin}
Suppose that $G$ is a compact quantum group and suppose that $u=((u_{ij}))$ is a $m$-dimensional representation of $G$ such that the $C^\ast$-algebra generated by $u_{ij}, i,j\in \{1,2,...,m\}$ in $C_m(G)$ is itself $C_m(G)$. Suppose further that up to unitary equivalence, there are only finitely many $n$-dimensional  representations of $G$. Then, $\Aut(G)$ and $\Aut_\chi(G)$ are compact Lie groups and $\out(G)$ is a finite group.
\end{theorem}

\begin{proof}
  Since $G$ is a compact matrix quantum group, it follows that
  $C_m(G)$ is separable, and therefore, $\Aut(C_m(G))$ is separable
  (in fact is Polish). It follows that $\Aut(G)$ is separable. Let now
  $\Gamma$ be a dense countable subgroup of $\Aut(G)$. Hence, we have
  a CQG dynamical system $(G,\Gamma)$. It follows then from the
  hypothesis that under the induced $\Gamma$-action on Irr$(G)$, the
  orbit $\{\Gamma[u]\}$ is finite (here $[u]$ denotes the equivalence
  class corresponding to the representation $u$ of $G$). Applying
  Proposition 6.7 of \cite{MP16}, we get then that the orbit of any
  element of Irr$(G)$ under the induced $\Gamma$-action is finite,
  whence it follows, using \Cref{Cpt}, that the CQG dynamical system
  $(G,\Gamma)$ is compact.

  However, since $\Gamma$ is dense in $\Aut(G)$, we then have that
  $\Aut(G)$ is compact, using Theorem 2.2 and Corollary 2.7 of
  \cite{Gr}. But, by \Cref{th.out}, $\out(G)$ is discrete for $G$, and
  since $\Aut(G)$ is compact, we have that $\out(G)$ is finite. Now,
  by \Cref{cor.lie}, $\Aut_\chi(G)$ is a compact Lie group and so,
  finiteness of $\out(G)$ implies that $\Aut(G)$ is also a compact Lie
  group. Hence, we are done.
\end{proof}

\begin{remark}
 Let us note that the condition mentioned above, of compact matrix quantum groups whose generating representation has some fixed dimension, say $m$, and there are only finitely many $m$-dimensional representations up to unitary equivalence, is satisfied by several examples, including the free unitary quantum groups, the free orthogonal quantum groups, the free permutation quantum groups, $q$-deformations of simple compact Lie groups, Rieffel deformations of simple compact Lie groups, etc.
\end{remark}



\section{CMQGs with infinitely generated fusion rings}\label{se.counter}

In this section we construct a compact matrix quantum group whose
complex fusion algebra is not finitely generated as an algebra. In
fact, the CQG algebra we consider is not only finitely generated
(hence corresponding to a CMQG), but even finitely {\it presented}.

Throughout the rest of this section $n\ge 2$ is understood to be an
arbitrary but fixed positive integer, and we denote $A=A_u(n)$. The CQG algebra $C$ that will provide the counterexample fits into an exact sequence
\begin{equation}\label{eq:ext}
        \begin{tikzpicture}[auto,baseline=(current  bounding  box.center)]
          \path[anchor=base] (0,0) node (c1) {$\bC$} +(2,0) node (z2) {$C(\bZ/2)$} +(4,0) node (c) {$C$} +(6,0) node (a) {$A$} +(8,0) node (c2) {$\bC$};
          \draw[->] (c1) -- (z2);
          \draw[->] (z2) -- (c);
          \draw[->] (c) -- (a);
          \draw[->] (a) -- (c2);          
        \end{tikzpicture}
\end{equation}
of Hopf algebras in the sense of \cite{ad}. The construction is also a
generalization of that of a bicrossed product as in \cite{fmp}, where
the compact group is $\bZ/2$, and the discrete group is the quantum
discrete dual to the free unitary group $U_n^+$ (i.e. the discrete
quantum group with group algebra $A$).

We construct $C$ as a bicrossed product $C(\bZ/2)\sharp A$ according
to the notation of \cite{ad}, following the recipe outlined in
\cite[Theorem 2.20]{ad}. According to (a very particular version of)
that result, all we need for such a construction to go through is a coaction
\begin{equation*}
  A\to A\otimes C(\bZ/2)
\end{equation*}
(i.e. an action of the group $\bZ/2$ on $A$) satisfying certain
conditions. Unpacking that theorem, it follows that the construction
does indeed result in a Hopf algebra structure on
$C=C(\bZ/2)\otimes A$ if

\begin{itemize}
\item $\bZ/2$ acts on $A$ by Hopf algebra automorphisms;
\item $C$ is equipped with the tensor product algebra structure;
\item $C$ is equipped with the smash coproduct \cite[(2.15)]{ad}.  
\end{itemize}
In reference to the third bullet point, we apply \cite[(2.15)]{ad} by
taking $\rho:A\to A\otimes C(\bZ/2)$ to be our coaction defining the
$(\bZ/2)$-action and $\tau:A\to \bC(\bZ/2)^{\otimes 2}$ to be the
trivial ``cycle''
\begin{equation*}
  A\ni a\mapsto \varepsilon(a)(1\otimes 1)\in C(\bZ/2)^{\otimes 2}. 
\end{equation*}
Moreover, \Cref{eq:ext} is a {\it central} extension, in the sense
that $C(\bZ/2)$ is contained in the center of $C$.

The action of $\bZ/2=\langle \gamma\rangle$ on $A$ will be defined as
\begin{equation}\label{eq:act}
  \gamma:u_{ij}\leftrightarrow u_{ij}^*
\end{equation}
at the level of generators, and extended in the only way possible to a
$*$-algebra automorphism (the relations of $A$ ensure that this is
indeed possible). The $*$-structure on $C$ is simply the tensor
product $*$-structure on $C(\bZ/2)\otimes A$, and is easily seen to
make $C$ into a CQG algebra in the sense of \cite{DK}.

We now need to understand the fusion ring $R(C)$ of $C$, in order to
show that it is not finitely generated. In fact, we will prove this
for the complexified version of $R$, namely
\begin{equation*}
  R_{\bC} = R_\bC(C) = \bC\otimes_{\bZ}R(C)
\end{equation*}
(the fact that $R_\bC$ is not finitely generated as a complex algebra
is a priori stronger than $R(C)$ not being finitely generated as a ring).

\Cref{pr.cat-desc} below describes the comodule category of $C$,
achieving the above-stated goal. Before giving the statement we recall
various results on the representation theory of $U_n^+$ from
\cite{Ban97} (see Th\'eor\`eme 1 therein) and make a few preparatory
remarks.

The standard $A$-comodule $V$ whose underlying matrix coalgebra is
spanned by the generators $u_{ij}$ and its dual $V^*$ (with matrix
coalgebra spanned by $u_{ij}^*$) generate the fusion ring of
$U_n^+$. The irreducible representations of $U_n^+$ are in bijection
with the words on the alphabet
\begin{equation*}
  \alpha=[V],\ \beta=[V^*]
\end{equation*}
(classes in the fusion semiring $R_+(A)$), in the sense that for every
word the corresponding tensor product of copies of $V$ and $V^*$ has a
unique simple summand that does not appear as a summand in a strictly
shorter word.

One consequence of this is that the fusion {\it ring} $R(A)$ is free
(as a ring) on the generators $\alpha$ and $\beta$. Indeed, the
description of the fusion rules in \cite[Th\'eor\`eme 1 (i)]{Ban97}
makes it clear that $R(A)$ admits a filtration by degree with
\begin{equation*}
  \mathrm{deg}(\alpha)=1=\mathrm{deg}(\beta)
\end{equation*}
such that the associated graded ring is indeed free, and hence the
canonical map from the free ring on two generators sending these
generators to $\alpha$ and $\beta$ respectively is an isomorphism of
rings.

Now, the automorphism $\gamma$ defined as in \Cref{eq:act} implements
an action by monoidal autoequivalences on the category $\cM=\cM^A$ of
finite-dimensional comodules by twisting via the automorphism $\gamma$. We denote this twisting functor by $\gamma^*$. The monoidal autoequivalence 
\begin{equation*}
  \gamma^*:\cM\to \cM
\end{equation*}
interchanges $V$ and $V^*$, and hence the resulting action on the
fusion semiring $R_+(A)$ interchanges $\alpha$ and $\beta$.

Moreover, the action of $\bZ/2$ on $\cM$ allows us to define the {\it
  equivariantization} $\cM^{\bZ/2}$ consisting of objects $W\in \cM$
equipped with isomorphisms $\gamma^*W\cong W$ satisfying certain
compatibility conditions (we refer the reader to \cite[$\S$3]{nat} and
references therein for background and details on the
equivariantization construction).

With this in place, we have

\begin{proposition}\label{pr.cat-desc}
  The category $\cM^C$ of finite-dimensional $C$-comodules is
  equivalent to the equivariantization $\cM^{\bZ/2}$
\end{proposition}
\begin{proof}
  This is simply a dual version of \cite[Proposition 3.5]{nat}, where
  the case of a {\it cocentral} extension of a group algebra is
  treated in the context of semisimple Hopf algebras and categories of
  modules rather than comodules are considered. The arguments can be
  repeated virtually verbatim.
\end{proof}

It will be important for us (for reasons that will become clear below)
to work with the fixed subring $R(A)^{\bZ/2}$, or rather the complex
fixed subalgebra $R_\bC(A)^{\bZ/2}$. The discussion above identifies
the action of $\gamma\in \bZ/2$ on $R_{\bC}(A)$ with the
generator-interchange map on the free algebra
$\bC\langle \alpha,\beta\rangle$.

As a consequence of \Cref{pr.cat-desc}, we have

\begin{corollary}\label{cor.surj-fus}
  The fusion semiring $R_+(C)$ surjects onto the invariant part
  $R_+(A)^{\bZ/2}$ of the fusion semiring of $A$ under the action that
  interchanges $\alpha,\beta\in R_+(A)$.
\end{corollary}
\begin{proof}
Consider the forgetful functor
\begin{equation*}
  \cat{forget}:\cM^{\bZ/2}\to \cM 
\end{equation*}
that forgets the equivariant structure on an object in
$\cM^{\bZ/2}$. The result follows by passing to the map
\begin{equation}\label{eq:r+forget}
  R_+(\cat{forget}):R_+(C)\to R_+(A)
\end{equation}
induced by $\cat{forget}$ at the level of fusion semirings, and
observing that {\it every} finite-dimensional $A$-comodule of the form
\begin{equation*}
  \gamma^* W\oplus W,\ W\in \cM
\end{equation*}
is in the essential image of $\cat{forget}$ and hence
\Cref{eq:r+forget} is surjective onto $R_+(A)^{\bZ/2}$.
\end{proof}

Finally, we obtain

\begin{corollary}\label{cor.not-fg}
  The complex fusion algebra $R_\bC(C)$ is not finitely generated. 
\end{corollary}
\begin{proof}
Indeed, \Cref{cor.surj-fus} ensures that the algebra in question surjects onto
\begin{equation*}
  R_\bC(A)^{\bZ/2}\cong \bC\langle \alpha,\beta\rangle^{\bZ/2}.
\end{equation*}
In turn, the latter algebra is not finitely generated by
\Cref{le.not-fg-aux} below.
\end{proof}

\begin{lemma}\label{le.not-fg-aux}
  Let $\bZ/2$ act on the free ring $F=\bC\langle \alpha,\beta\rangle$ by
  interchanging the generators $\alpha$ and $\beta$. Then, the
  invariant subalgebra $\bC\langle \alpha,\beta\rangle^{\bZ/2}$ is not
  finitely generated over $\bC$.
\end{lemma}
\begin{proof}
  For a word $w$ consisting of letters $\alpha$ and $\beta$, denote by
  $w^\gamma$ the result of applying the transformation
  $\alpha\leftrightarrow \beta$ to $w$ and by $w^*$ the sum
  $w+w^\gamma$. As a vector space, $F^{\bZ/2}$ has a basis consisting
  of all distinct elements of $F$ of the form $w^*$.

  Note that $F$ has a grading by
  $\mathrm{deg}(\alpha)=\mathrm{deg}(\beta)=1$. Now, assuming by
  contradiction that $F^{\bZ/2}$ is finitely generated, we may as well
  suppose that the set of elements $w^*$ with $\mathrm{deg}(w)\le k$
  generates for some $k>1$. We now seek to achieve a contradiction by
  showing that the subalgebra generated by such $w^*$ cannot contain
  the degree-$(k+1)$ component $(F^{\bZ/2})_{k+1}$.

  To this aim, let $w^*\in F^{\bZ/2}$ be an arbitrary element with
  $\mathrm{deg}(w)=k+1$, and assume it can be written as a linear combination of products
  \begin{equation}\label{eq:prod-t}
    w_1^* \cdots w_t^*,\ \sum_{i=1}^t \deg(w_i) = k+1,\ \deg(w_i)>0. 
  \end{equation}
Note furthermore that a product \Cref{eq:prod-t} is contained in the span of products
\begin{equation}
  \label{eq:prod-2}
  w_1^* w_2^*,\ \deg(w_1)+\deg(w_2) = k+1,\ \deg(w_i)>0,
\end{equation}
so we may as well assume we are only dealing with these latter ones
(i.e. $t=2$ in all terms \Cref{eq:prod-t} in the decomposition of
$w^*$).

We construct a graph having 
\begin{equation*}
  V_{k+1} = \{w^*\ |\ \deg(w)=k+1\}
\end{equation*}
as vertices by placing an edge between $w^*$ and $v^*$ whenever 
\begin{equation*}
  w_1^*w_2^* = w^*+v^*
\end{equation*}
for some $w_i$ as in \Cref{eq:prod-2}. The graph $\Gamma_{k+1}$ in
question is (isomorphic to) the hypercube with $2^k$ vertices, and a
linear combination of elements $w^*$ for $\deg(w)=k+1$ can be regarded
as a linear combination of vertices of this hypercube. Now,
$\Gamma_{k+1}$ is bipartite and we will henceforth refer to the
vertices in the two parts as black and white respectively. Each
product \Cref{eq:prod-2} assigns equal coefficients to two vertices of
opposite colors and coefficient zero to every other vertex. In other
words, the linear invariant
\begin{equation}\label{eq:bw}
  (\text{sum of coefficients of black vertices})-(\text{sum of coefficients of white vertices})
\end{equation}
vanishes for linear combinations of products of the form
\Cref{eq:prod-2}. On the other hand, a single element $w^*$,
$\deg(w)=k+1$, when regarded as a linear combination of vertices of
$\Gamma_{k+1}$, assigns coefficient $1$ to a single vertex and zero to
all other vertices. In conclusion, its invariant \Cref{eq:bw} is
$\pm 1$, contradicting the assumption that $w^*$ can be written as a
linear combination of \Cref{eq:prod-2}.
\end{proof}

Finally, we can phrase the conclusion of the present section as follows:

\begin{corollary}\label{cor.counter}
  There exist compact quantum groups whose underlying CQG algebra is finitely presented and whose complexified fusion algebra is not finitely generated. 
\end{corollary}
\begin{proof}
  $C$ is such an object: as an algebra it is simply the tensor product
  $C(\bZ/2)\otimes A$, and is hence finitely presented because $A$
  is. On the other hand, the non-finite generation of the fusion
  algebra is \Cref{cor.not-fg} above.
\end{proof}


\bibliography{fus}{}

\def\polhk#1{\setbox0=\hbox{#1}{\ooalign{\hidewidth
  \lower1.5ex\hbox{`}\hidewidth\crcr\unhbox0}}}
\begin{thebibliography}{10}

\bibitem{ad}
N.~Andruskiewitsch and J.~Devoto.
\newblock Extensions of {H}opf algebras.
\newblock {\em Algebra i Analiz}, 7(1):22--61, 1995.

\bibitem{Av}
D.~Avitzour.
\newblock Noncommutative topological dynamics. ii.
\newblock {\em Trans. Amer. Math. Soc.}, 282(1):121--135, 1982.

\bibitem{Ban97}
Teodor Banica.
\newblock Le groupe quantique compact libre {${\rm U}(n)$}.
\newblock {\em Comm. Math. Phys.}, 190(1):143--172, 1997.

\bibitem{Ban05}
Teodor Banica.
\newblock Quantum automorphism groups of small metric spaces.
\newblock {\em Pacific J. Math.}, 219(1):27--51, 2005.

\bibitem{bs}
George~M. Bergman and Saharon Shelah.
\newblock Closed subgroups of the infinite symmetric group.
\newblock {\em Algebra Universalis}, 55(2-3):137--173, 2006.
\newblock Special issue dedicated to Walter Taylor.

\bibitem{BhoGos09}
Jyotishman Bhowmick and Debashish Goswami.
\newblock Quantum group of orientation-preserving {R}iemannian isometries.
\newblock {\em J. Funct. Anal.}, 257(8):2530--2572, 2009.

\bibitem{BSS15}
Jyotishman Bhowmick, Adam Skalski, and Piotr~M. So\l~tan.
\newblock Quantum group of automorphisms of a finite quantum group.
\newblock {\em J. Algebra}, 423:514--537, 2015.

\bibitem{Bic03}
Julien Bichon.
\newblock Quantum automorphism groups of finite graphs.
\newblock {\em Proc. Amer. Math. Soc.}, 131(3):665--673 (electronic), 2003.

\bibitem{DK}
Mathijs~S. Dijkhuizen and Tom~H. Koornwinder.
\newblock C{QG} algebras: a direct algebraic approach to compact quantum
  groups.
\newblock {\em Lett. Math. Phys.}, 32(4):315--330, 1994.

\bibitem{fmp}
P.~{Fima}, K.~{Mukherjee}, and I.~{Patri}.
\newblock {On compact bicrossed products}.
\newblock {\em arXiv:1504.00092}, March 2015.

\bibitem{Gos12}
D.~{Goswami}.
\newblock {Existence of quantum isometry group for a class of compact metric
  spaces}.
\newblock {\em arXiv:1205.6099}, May 2012.

\bibitem{gj}
D.~{Goswami} and S.~{Joardar}.
\newblock {Rigidity of action of compact quantum groups on compact, connected
  manifolds}.
\newblock {\em arXiv:1309.1294}, September 2013.

\bibitem{Gr}
W.L. Green.
\newblock Topological dynamics and $c^\ast$-algebras.
\newblock {\em Trans. Amer. Math. Soc.}, 210:107--121, 1975.

\bibitem{hlm}
Paul~R. Halmos.
\newblock On automorphisms of compact groups.
\newblock {\em Bull. Amer. Math. Soc.}, 49:619--624, 1943.

\bibitem{Han93}
David Handelman.
\newblock Representation rings as invariants for compact groups and limit ratio
  theorems for them.
\newblock {\em Internat. J. Math.}, 4(1):59--88, 1993.

\bibitem{Iwa}
K.~Iwasawa.
\newblock On some types of topological groups.
\newblock {\em Ann. of Math.(2)}, 50:507--558, 1949.

\bibitem{jw95}
Wojciech Jaworski.
\newblock Strong approximate transitivity, polynomial growth, and spread out
  random walks on locally compact groups.
\newblock {\em Pacific J. Math.}, 170(2):517--533, 1995.

\bibitem{jw12}
Wojciech Jaworski.
\newblock Contraction groups, ergodicity, and distal properties of
  automorphisms of compact groups.
\newblock {\em Illinois J. Math.}, 56(4):1023--1084, 2012.

\bibitem{KSS16}
Pawel {Kasprzak}, Adam {Skalski}, and Piotr~M. {So\l tan}.
\newblock {The canonical central exact sequence for locally compact quantum
  groups}.
\newblock {\em arXiv:1508.02943}, 2016.

\bibitem{KSW15}
Pawel Kasprzak, Piotr~M. Soltan, and Stanis\l aw~L. Woronowicz.
\newblock Quantum automorphism groups of finite quantum groups are classical.
\newblock {\em J. Geom. Phys.}, 89:32--37, 2015.

\bibitem{ktz}
Yitzhak Katznelson.
\newblock Ergodic automorphisms of {$T^{n}$} are {B}ernoulli shifts.
\newblock {\em Israel J. Math.}, 10:186--195, 1971.

\bibitem{ktch-sch}
Bruce Kitchens and Klaus Schmidt.
\newblock Automorphisms of compact groups.
\newblock {\em Ergodic Theory Dynam. Systems}, 9(4):691--735, 1989.

\bibitem{ktch}
Bruce~P. Kitchens.
\newblock Expansive dynamics on zero-dimensional groups.
\newblock {\em Ergodic Theory Dynam. Systems}, 7(2):249--261, 1987.

\bibitem{MP16}
Kunal Mukherjee and Issan Patri.
\newblock Automorphisms of compact quantum groups.
\newblock {\em Submitted}, 2016.

\bibitem{nat}
Sonia Natale.
\newblock Hopf algebra extensions of group algebras and {T}ambara-{Y}amagami
  categories.
\newblock {\em Algebr. Represent. Theory}, 13(6):673--691, 2010.

\bibitem{iss}
Issan Patri.
\newblock Normal subgroups, center and inner automorphisms of compact quantum
  groups.
\newblock {\em Internat. J. Math.}, 24(9):1350071, 37, 2013.

\bibitem{htz}
Federico Rodriguez~Hertz.
\newblock Stable ergodicity of certain linear automorphisms of the torus.
\newblock {\em Ann. of Math. (2)}, 162(1):65--107, 2005.

\bibitem{Se68}
Graeme Segal.
\newblock The representation ring of a compact lie group.
\newblock {\em Inst. Hautes Études Sci. Publ. Math.}, 34(1):113--128, 1968.

\bibitem{DaeWan96}
Alfons Van~Daele and Shuzhou Wang.
\newblock Universal quantum groups.
\newblock {\em Internat. J. Math.}, 7(2):255--263, 1996.

\bibitem{Wan95_1}
Shuzhou Wang.
\newblock Free products of compact quantum groups.
\newblock {\em Comm. Math. Phys.}, 167(3):671--692, 1995.

\bibitem{Wan95_2}
Shuzhou Wang.
\newblock Tensor products and crossed products of compact quantum groups.
\newblock {\em Proc. London Math. Soc. (3)}, 71(3):695--720, 1995.

\bibitem{Wan98}
Shuzhou Wang.
\newblock Quantum symmetry groups of finite spaces.
\newblock {\em Comm. Math. Phys.}, 195(1):195--211, 1998.

\bibitem{Wor87}
S.~L. Woronowicz.
\newblock Compact matrix pseudogroups.
\newblock {\em Comm. Math. Phys.}, 111(4):613--665, 1987.

\bibitem{Wor98}
S.~L. Woronowicz.
\newblock Compact quantum groups.
\newblock In {\em Sym\'etries quantiques ({L}es {H}ouches, 1995)}, pages
  845--884. North-Holland, Amsterdam, 1998.

\bibitem{ydv}
Manoj~K. Yadav.
\newblock Class preserving automorphisms of finite {$p$}-groups: a survey.
\newblock In {\em Groups {S}t {A}ndrews 2009 in {B}ath. {V}olume 2}, volume 388
  of {\em London Math. Soc. Lecture Note Ser.}, pages 569--579. Cambridge Univ.
  Press, Cambridge, 2011.

\end{thebibliography}
\bibliographystyle{plain}

\Addresses

\end{document}